\documentclass[letterpaper, 10 pt, conference]{ieeeconf}  

\IEEEoverridecommandlockouts                              
\usepackage{amsfonts,amsmath,amssymb,bm}
\usepackage{color,cases,float}
\usepackage{graphicx}
\usepackage{hyperref}
\usepackage{psfrag}
\usepackage{eucal}
\usepackage{algorithm}
\usepackage{overpic}
\usepackage{subfigure}
\usepackage{url}
\usepackage[all,cmtip]{xy}
\usepackage{algorithm}
\usepackage[noend]{algpseudocode}

\definecolor{darkblue}{rgb}{0.0,0.0,0.6}
\hypersetup{colorlinks,breaklinks,
	linkcolor=darkblue,urlcolor=darkblue,
	anchorcolor=darkblue,citecolor=darkblue}
\newtheorem{assumption}{Assumption}
\newtheorem{definition}{Definition}
\newtheorem{lem}{Lemma}
\newtheorem{rem}{Remark}
\newtheorem{prop}{Proposition}
\newtheorem{theorem}{Theorem}

\newtheorem{example}{Example}

\newcommand{\tat}[1]{\normalsize{{\color{black}\ #1}}}
\newcommand{\E}{\mathrm{E}}

\newcommand{\R}{\mathrm{Re}}
\newcommand{\Proj}{\mathrm{Proj}}

\def\R{\mathbb{R}}

\def\bx{\boldsymbol{x}}

\def\ab{\boldsymbol{a}}

\def\mb{\boldsymbol{m}}

\def\s{\sigma}
\def\g{\gamma}

\def\bmu{\boldsymbol{\mu}}

\def\tat#1{{\color{black}#1}}

\def\R{\mathbb{R}}

\def\xb{\boldsymbol{x}}

\def\ba{\boldsymbol{a}}

\def\bmu{\boldsymbol{\mu}}

\def\Mb{\boldsymbol{M}}

\def\Gb{\boldsymbol{G}}

\def\Rb{\boldsymbol{R}}
\def\Ab{\boldsymbol{A}}

\def\bxi{\boldsymbol{\xi}}
\def\r{~}
\def\BibTeX{{\rm B\kern-.05em{\sc i\kern-.025em b}\kern-.08em
		T\kern-.1667em\lower.7ex\hbox{E}\kern-.125emX}}

	\title{Convergence Rate of Learning a Strongly Variationally Stable Equilibrium}
	\author{Tatiana Tatarenko, Maryam Kamgarpour  \IEEEmembership{Member, IEEE}
\thanks{This work has been accepted for publication in the proceedings of the European Control Conference 2024.}
    \thanks{T. Tatarenko is with TU Darmstadt, Control Methods and Intelligent Systems Lab (e-mail: tatiana.tatarenko@tu-darmstadt.de).}
		\thanks{M. Kamgarpour  is with EPFL Sycamore Lab (e-mail: maryam.kamgarpour@epfl.ch). }}

\begin{document}

\maketitle
\thispagestyle{empty}
\pagestyle{empty}

\begin{abstract}
	We derive the rate of convergence to the  strongly variationally stable Nash equilibrium in a convex game, for  a zeroth-order learning algorithm. Though we do not assume strong monotonicity of the game, our rates for the one-point feedback, $O\left(\frac{Nd}{t^{1/2}}\right)$, and for the two-point feedback, $O\left(\frac{N^2d^2}{t}\right)$,  match the best known rates for strongly monotone games \tat{under zeroth-order information}.
	\end{abstract}

\section{Introduction}
\label{sec:intro}


Game-theoretic learning under zeroth-order information consists in deriving an algorithm for each player that uses only evaluations of the player's cost function.
This setting arises in applications in which each agent does not know the functional form of her objective or cannot readily compute its gradient, due to complex dependence on other players actions. For example, price functions in electricity markets depend on consumption/production of all agents in non-trivial way \cite{elmark}, travel times in a routing or transportation network depend on routes taken by other agents and the link capacities \cite{netrout}. In contrast, the agents can query the objectives at their joint chosen action, and obtain zeroth-order information  (cost function evaluations).  

The works \cite{bravo2018bandit,tatarenko2019learning, tat_kam_TAC}  proposed zeroth-order learning algorithms  for games over continuous action sets. The underlying idea in the above algorithms is developing a randomized sampling technique to estimate gradients of players' cost functions using the zeroth-order information, and to then use the estimated gradient in a stochastic gradient descent scheme.  In general, convergence in zeroth-order learning is slow due to the high variance of the gradient estimators. Hence, it is relevant to establish  optimal rates of convergence for this class of problems. Our goal in this paper to estimate rate of convergence for zeroth-order learning algorithms in a specific class of games. 

While convergence rates in zeroth-order convex optimization have been well-explored, less work has been dedicated to deriving convergence rates for zeroth-order learning in convex games. Under  strong monotonicity assumption on the game pseudo-gradient, past work  derived a rate of $O\left(\frac{1}{t^{1/3}}\right)$  for the proposed learning algorithms  \cite[Theorem 5.2]{bravo2018bandit}, \cite[Theorem 3]{tat_kam_TAC}. Recently, \cite{drusvyatskiy2022improved}  demonstrated that the rate of  $O\left(\frac{1}{t^{1/3}}\right)$  for the algorithm proposed in  \cite{bravo2018bandit} is suboptimal. Indeed, a refined analysis technique of the same algorithm along with suitable choices of the stepsize and sampling distribution can ensure $O\left(\frac{1}{t^{1/2}}\right)$ \cite{drusvyatskiy2022improved}. Independently, \cite{arxivTatKam_Feb2022} showed that the rate derived in \cite[Theorem 3]{tat_kam_TAC} for their proposed algorithm can be improved to $O\left(\frac{1}{t^{1/2}}\right)$. The rate of $O\left(\frac{1}{t^{1/2}}\right)$ appears optimal  as it matches the lower bound for the class of zeroth-order strongly convex smooth optimization under one-point feedback (a subset of the class of strongly convex games considered in the works mentioned above) \cite{RechtNIPS2012}. 
Both \cite{drusvyatskiy2022improved,arxivTatKam_Feb2022}, require strong monotonicity of the game pseudo-gradient in their rate analysis. 
 
 A major recent  interest in learning equilibria in convex games is on relaxing the requirement of monotonicity on the game pseudo-gradient. \tat{To this end, some works have relaxed the assumption of strongly monotone pseudo-gradients and considered games with pseudo-gradients which are restricted strongly monotone with respect to a Nash equilibrium \cite{Franci21, Lan2022, TatNedShi21}. However, the convergence rate of zeroth-order learning algorithms for such games have not been addressed. } 
Further relaxing the monotonicity requirements, the work \cite{mertikopoulos2019learning}  considers the so-called  (local/global) \emph{variational stability} of an equilibrium. While the game monotonicity implies variational stability of the equilibria,  an equilibrium can be variationally stable (VS) even when the game psuedo-gradient is not (strongly) monotone or \tat{restricted strongly monotone}, see Examples 1 and 2\footnote{In contrast to strong monotonicity, which is a property of the game pseudo-gradient only,  the conditions of variational stability and restricted strong monotonicity entail properties of the pseudo-gradient \emph{and a  Nash equilibrium} of the game. Accordingly, to establish these latter properties theoretically, one requires knowledge regarding the equilibrium point. This is a  trade-off allowing for establishing convergence and its rate in games with non-(strongly) monotone pseudo-gradients.}. 

It has been shown that the existence of a global strongly VS Nash equilibrium is sufficient for convergence of the first-order learning algorithms proposed in \cite{bravo2018bandit,mertikopoulos2019learning}. Given stochastic first-order information,  \cite{mertikopoulos2019learning}  derived a convergence rate to the strongly VS Nash equilibrium. This rate was in terms of ergodic average of the sequence played actions, a weaker notion of convergence than the  last iterate of the played actions derived in \cite{bravo2018bandit,tat_kam_TAC,drusvyatskiy2022improved,arxivTatKam_Feb2022}. Relaxing from strong to mere variational stability of an equilibria, the work  in \cite{gao2021second}  proposed an algorithm that converges to an interior mere VS equilibrium of a convex game under exact first-order feedback, i.e. knowledge of game pseudo-gradient, and characterized its convergence rate. Building on this, \cite{gao2022bandit} addressed learning of a mere VS equilibrium with zeroth-order information. However, convergence rates were not established in this work. 

Summarizing the above, the problem of characterizing the  rate of convergence of the iterates to the strongly variationally stable equilibrium under zeroth-order feedback  to our knowledge was not addressed. Addressing this gap, our  contributions are as follows. 
\begin{itemize}
\item  We derive the convergence rate of the zeroth-order gradient play to the strongly VS Nash equilibrium of a convex game using a one-point feedback as $O\left(\frac{1}{t^{1/2}}\right)$. This appears to be the best rate since it meets the known best bound in a subclass containing strongly monotone games given the same information setting \cite{drusvyatskiy2022improved, RechtNIPS2012, arxivTatKam_Feb2022}. 
\item We consider a two-point zeroth-order feedback model, motivated by the rate improvement achieved with the  two-point feedback model in the zeroth-order optimization literature~\cite{Duchi2015}. By adapting our randomized gradient estimation approach, we also improve the rate of convergence to the order of $O\left(\frac{1}{t}\right)$. 
\end{itemize}

The rest of the paper is organized as follows. In Section~\ref{sec:problem} we formulate the problem of payoff-based learning and state the assumptions on the considered class of games. In Section~\ref{sec:procedure}, we detail the proposed payoff-based approach in the one-point and two-point setting for gradient estimations. In Section~\ref{sec:rate} we state the main result on the convergence rate of the proposed algorithm and provide its proof. Section~\ref{sec:simulation} presents a simulation of the procedure under consideration. In Section~\ref{sec:conclusion} we conclude the paper.

\textbf{Notations.}
The set $\{1,\ldots,N\}$ is denoted by $[N]$. We consider real normed space $\R^d$. The column vector $\xb\in\R^d$ is denoted by $\xb = (x^1,\ldots, x^d)$. We use superscripts to denoted coordinates of vectors and the player-related functions and sets. We use the subscript $j$ which takes values $j\in\{1,2\}$ to differentiate between particular terms in the proposed algorithm for one and two-point feedback models, respectively.
For any function $f:K\to\R$, $K\subseteq\R^d$, $\nabla_{x^i} f(\xb) = \frac{\partial f(\xb)}{\partial x^i}$ is the partial derivative taken in respect to the $x^i$th variable (coordinate) in the vector argument $\xb\in\R^d$.
 We use $\langle \cdot,\cdot\rangle$ to denote the inner product in $\R^d$.
We use $\|\cdot\|$ to denote the Euclidean norm induced by the standard dot product in $\R^d$. 
A mapping $g:\R^d\to \R^d$ is said to be \emph{strongly monotone} on $Q\subseteq \R^d$ with the constant $\eta$, if for any $u, v\in Q$, $\langle g(u)-g(v), u - v\rangle \ge \eta\|u - v\|^2$; \emph{strictly monotone}, if the strict inequality holds for $\eta=0$ and $u\neq v$, and \emph{merely monotone} if the inequality holds for $\eta=0$.
We use $\Proj_{\Omega}{v}$ to denote the projection of $v\in E$ to a set $\Omega\subseteq E$.
The mathematical expectation of a random value $\xi$ is denoted by $\E\{\xi\}$. Its conditional expectation in respect to some $\sigma$-algebra $\EuScript F$ is denoted by $\E\{\xi|\EuScript F\}$.
We use the big-$O$ notation, that is, the function $f(x): \R\to\R$ is $O(g(x))$ as $x\to a$ for some $a\in\R$, i.e. $f(x)$ = $O(g(x))$ as $x\to a$, if $\lim_{x\to a}\frac{|f(x)|}{|g(x)|}\le K$ for some positive constant $K$. We use the little-$o$ notation, that is, the function $f(x): \R\to\R$ is $o(g(x))$ as $x\to a$ for some $a\in\R$, i.e. $f(x)$ = $o(g(x))$ as $x\to a$, if $\lim_{x\to a}\frac{|f(x)|}{|g(x)|} = 0$.

\section{Game setup and \tat{Zeroth-Order} Algorithm}
\label{sec:problem}
\allowdisplaybreaks
Consider a game $\Gamma = \Gamma (N, \{A^i\}, \{J^i\})$ with $N$ players, the sets of players' actions $A^i\subseteq \R^d$, $i\in[N]$, and the cost (objective) functions $J^i:\Ab\to\R$, where $\Ab = A^1\times\ldots\times A^N$ denotes the set of joint actions\footnote{For notation simplicity, we assume the dimension of each action set to be $d$. The algorithm and analysis readily generalize to the case of different dimensions $d^i$, $i\in[N]$.}. Thus, each joint action is a vector $\ba = (\ba^1,\ldots,\ba^N)\in\Ab\subseteq \R^{Nd}$, where $\ba^i = (a^{i,1},\ldots,a^{i,d})\in\R^d$. We  use the notation $\ba = (\ba^i,\ba^{-i})$, where $\ba^{-i}$ is actions of players not including player $i$.

\begin{definition}\label{def:NE}
	A vector $\ba^* = (\ba^{*1},\ldots,\ba^{*N})\in\Ab$ is called a \emph{Nash equilibrium} if for any $i\in[N]$ and $\ba^i\in A^i$
	\[J^i(\ba^{i*},\ba^{-i*})\le J^i(\ba^{i},\ba^{-i*}).\]
\end{definition}
We restrict the class of games as follows.
\begin{assumption}\label{assum:convex}
	The game under consideration is \emph{convex}. Namely, for all $i\in[N]$ the set $A^i$ is convex and closed, the cost function $J^i(\ba^i, \ba^{-i})$ is defined on $\R^{Nd}$, continuously differentiable in $\ba$ and convex in $\ba^i$ for  fixed $\ba^{-i}$.
\end{assumption}

\begin{assumption}\label{assum:compact}
	The action sets $A^i$, $i \in [N]$, are compact.
\end{assumption}
Note that Assumptions~\ref{assum:convex} and~\ref{assum:compact} together imply the existence of a Nash equilibrium in the game $\Gamma$ \cite{FaccPang2}. In a convex game, the Nash equilibrium can be characterized through the so-called \emph{pseudo-gradient} of the game defined below.

\begin{definition}\label{def:pg}
	The mapping $\Mb:\R^{Nd}\to\R^{Nd}$, referred to as the \emph{pseudo-gradient} of the game $\Gamma (N, \{A^i\}, \{J^i\})$, is defined by
	\begin{align*}
		\nonumber
		\Mb(\ba) &= (\nabla_{\ba^i} J^i(\ba^i, \ba^{-i}))_{i=1}^N=(\Mb^1(\ba), \ldots, \Mb^N(\ba))^{\top},\\ \nonumber
		&\mbox{where }\Mb^i(\ba) = (M^{i,1}(\ba), \ldots, M^{i,d}(\ba))^{\top}, \\
		\nonumber
		&\quad M^{i,k}(\ba)= \frac{\partial J^i(\ba)}{\partial a^{i,k}},\quad \ba\in\Ab, \quad i\in[N], \quad k\in[d].
	\end{align*}	
\end{definition}
In a convex game, i.e. under Assumption~\ref{assum:convex}, $\ab^*$ is a Nash equilibrium if and only if $\langle\Mb(\ab^*),\ab-\ab^*\rangle \geq 0, \; \forall \ab \in \Ab$ (see, for example, \cite{FaccPang2}). However, this characterization alone is not sufficient to ensure convergence of learning algorithms using the idea of a (stochastic) gradient descent approach (called also gradient play) in convex games.
In particular, in most past work on learning algorithms certain structural assumptions on the game pseudo-gradient, such as strong/strict monotonicity, or assumptions on the Nash equilibrium such as variational stability, are required to prove convergence of algorithms. 

\begin{definition}\label{def:VS}
	A Nash equilibrium $\ab^*$ 
	is globally \emph{$\nu$-strongly variationally stable (SVS)}, if $\langle\Mb(\ab),\ab - \ab^*\rangle \ge \nu\|\ab - \ab^*\|^2$ for any $\ab\in\Ab$ and some $\nu>0$. 
\end{definition}
In the definition above, if the inequality holds with $\nu = 0$, then the Nash equilibrium at $\ab^*$ is referred to as globally merely variationally stable. On the other hand, if the above properties hold only on a neighborhood $\mathbf{D} \subset \Ab$, then the Nash equilibrium is locally (strongly/merely) VS. 

\begin{assumption}\label{assum:CG_grad}
	The Nash equilibrium in $\Gamma$ is globally $\nu$-strongly variationally stable with the constant $\nu$.
\end{assumption}
If a game has a strongly variationally stable (SVS) Nash equilibrium, then the Nash equilibrium is unique \cite[Proposition 2.5]{mertikopoulos2019learning}. Furthermore,  if a game has a strongly monotone pseudo-gradient, then its unique Nash equilibrium is strongly variationally stable. However, the converse statement is not true. The example below illustrates these definitions. 

\begin{example}\label{example:3player}
Consider a 3-player game, where each player's action set is $[-1,2] \subset \R$. The cost of player $i$, for $i\in\{1,2,3\}$ is $J^i(a^1,a^2,a^3) = a^1a^2a^3 + (a^i)^2$. Hence, the game pseudo-gradient is given by $\Mb(\ab) = (a^2a^3+2a^1, a^1a^3+2a^2, a^1a^2+2a^3)^{\top}$. It can be verified that there exists a Nash equilibrium at $\ab^*=(0,0,0)$, since $\Mb(\ab^*)=0$. Furthermore, this Nash equilibrium is globally strongly VS with $\nu = 1/2$, since $\langle\Mb(\ab),\ab-\ab^*\rangle = 3a^1a^2a^3 + 2\sum_{i=1}^3 (a^i)^2 \geq \frac{\| \ab-\ab^*\|^2}{2}$ for any $\ab\in \Ab = [-1,2]^3$. Notice that the game is not monotone since the Jacobian of the $\Mb(\ab)$, given as {\small $$\nabla \Mb(\ab)=\begin{pmatrix}
2 & a^3 & a^2\\
a^3 & 2 & a^1 \\
a^2 & a^1 & 2 \\
\end{pmatrix}$$}has a negative eigenvalue for $\ab = (2,1,2)$.  Furthermore, the restriction of the game to the action set $[-1,0]$ results in the same unique globally stable Nash equilibrium, but this time, this equilibrium will be on the boundary. 

\end{example}
For further examples of games exhibiting variationally stable equilibria, please, see examples in \cite{mertikopoulos2019learning,gao2022bandit}. For examples of games from  telecommunication or adversarial learning domains with VS Nash equilibria, please, see \cite{gao2022bandit}.

\tat{\begin{rem} Recent works have addressed convergent procedures in games with \emph{restricted strongly monotone} pseudo-gradients~\cite{Franci21, Lan2022, TatNedShi21}. This property is formulated as follows: The pseudo-gradient $\Mb$ is called restricted strongly monotone in the game $\Gamma$ possessing a Nash equilibrium $\ab^*$, if $\langle\Mb(\ab)-\Mb(\ab^*), \ab-\ab^*\rangle\ge \nu \|\ab-\ab^*\|^2$ for some $\nu>0$ and any $\ab\in\Ab$. If some game satisfies this condition, the Nash equilibrium is unique \cite{TatNedShi21}. It can be seen that, due to the inequality $\langle\Mb(\ab^*), \ab-\ab^*\rangle\ge 0$ holding for any $\ab\in\Ab$, the Nash equilibrium in a game with a restricted strongly monotone pseudo-gradient is necessarily strongly variationally stable. However, as Example~\ref{example:2player} below demonstrates, existence of a strongly variationally stable Nash equilibrium does not imply that the pseudo-gradient is restricted strongly monotone. Thus, games with {restricted strongly monotone} pseudo-gradients are a subclass of games with strongly variationally stable Nash equilibria considered in this paper.
\end{rem}
}

\tat{\begin{example}\label{example:2player}
Consider a 2-player game, where each player's action set is $[0,1] \subset \R$. The cost of each player $i$, for $i\in\{1,2\}$, is $J^i(a^1,a^2) =\frac{1}{4}((a^1)^2 + (a^2)^2) - \frac{1}{4}a^1a^2 + 2\sqrt{1+a^1} +  2\sqrt{1+a^2}$. The unique minimizer of the equal cost functions and, thus, the unique Nash equilibrium of the game, is $\ab^*=(0,0)$. The game pseudo-gradient is given by $\Mb(\ab) = (\frac{1}{2} a^1 -\frac{1}{4}a^2+ \frac{1}{\sqrt{1+a^1}}, \frac{1}{2} a^2 -\frac{1}{4}a^1 + \frac{1}{\sqrt{1+a^2}})^{\top}$. It can be verified that $\langle\Mb(\ab), \ab-\ab^*\rangle \ge \nu\|\ab-\ab^*\|^2$ with $\nu=\frac{1}{4}$. Thus, the game possesses the unique globally strongly VS Nash equilibrium. However, the pseudo-gradient of the game is not restricted strongly monotone, as for $\ab = (1,1)\in \Ab$
\begin{align*}
    \langle\Mb(\ab)-\Mb(\ab^*), \ab-\ab^*\rangle &= \frac{1}{2}(a^1)^2 + \frac{1}{2}(a^2)^2 - \frac{1}{2}a^1a^2\cr
    &\,+ \left(\frac{a^1}{\sqrt{1+a^1}} + \frac{a^2}{\sqrt{1+a^2}}\right)\cr
    &\,- a^1- \left.a^2\right|_{\ab = (1,1)}<0.
  \end{align*}
\end{example}}


\section{Payoff-based Learning algorithm}\label{sec:procedure}
The steps of the procedure run by each player are summarized in Algorithm \ref{alg:algorithm1}. In particular, the one-point approach has already been proposed in \cite{tat_kam_TAC, arxivTatKam_Feb2022}. However, its convergence properties established in the above works  was only for the case of strongly monotone games. 
\setlength{\textfloatsep}{15pt}
\begin{algorithm}[t!]
	\caption{One-point and two-point zeroth-order  algorithm for learning  Nash equilibria }\label{alg:algorithm1}
	\begin{algorithmic}[!t]
		\Require Action set $\Ab^i \subset \R^d$, the sequences $\{\sigma_t\},  \{\gamma_t\}$, initial state $\bmu^i(0)$.
		\For {$t = 0,1, \ldots$}
		\State  Sample a query point $\bxi^i(t)$ according to probability density \eqref{eq:density}.
		
		\noindent{\color{blue} { \footnotesize\ttfamily  /* Simultaneously and similarly,  other players sample their  query points $\bxi^{-i}(t)$.   /* }}
		\State One-point scheme: Perform the one-point gradient estimate according to \eqref{eq:est_Gd}.		\State Two-point scheme:  Observe additional  ${J^i}^0(t) = J^i(\bmu^1(t),\ldots,\bmu^N(t))$ and perform the two-point gradient estimate according to \eqref{eq:est_Gd2}.
		\State Update the state according to \eqref{eq:alg}.		
		\noindent{\color{blue} { \footnotesize\ttfamily  /* Simultaneously and similarly,  other players update their states $\bmu^{-i}(t+1)$.  /* }}
		\EndFor
		
	\end{algorithmic}
\end{algorithm}

\subsubsection{Algorithm iterates}
Let us denote by $\mb^i_j$, $j\in\{1,2\}$, some estimate of $\Mb^i$ in the pseudo-gradient of the game (see Definition \ref{def:pg}). Here, $j=1$ denotes the one-point and $j=2$ denotes the two-point procedure estimate, respectively, and will be detailed in the next subsection. The proposed method  to update  player $i$'s  so-called state $\bmu^i$ is as follows:
\begin{align}
	\label{eq:alg}
	\bmu^i(t+1)=\Proj_{A^i}[\bmu^i(t)-\gamma_t\mb^i_j(t)],
\end{align}
where $\bmu^i(0)\in \R^{Nd}$ is an arbitrary finite value and $\g_t$ is the step size or the learning rate. The step size $\gamma_t$ needs to be chosen based on the bias and variance of the pseudo-gradient estimates $\mb^i_j$. The term $\mb^i_j(t)$, $j\in\{1,2\}$, is obtained using the payoff-based feedback as described below.

\begin{rem}
While our algorithm is similar to \cite{bravo2018bandit,tat_kam_TAC}, the reason for being able to establish the stronger result compared to these past works is our new analysis technique. In particular, due to the lack of strong monotonicity, we develop a new  approach to estimate the distance between the algorithm iterates and the Nash equilibrium. This approach is based on proving that the Nash equilibrium $\ab^*$ stays ``almost" strongly variationally stable with respect to the pseudo-gradient in the \emph{mixed strategies},  a property we establish in Proposition~\ref{prop:strMon}.
\end{rem}

\subsubsection{Gradient estimation in one and two-point settings}
We estimate the unknown gradients using the randomizing sampling technique. In particular, we use the Gaussian distribution for sampling inspired by \cite{Thatha,NesterovSpokoiny}. Since this distribution has an unbounded support, we need the following assumption on the cost functions' behavior at infinity.
\begin{assumption}	\label{assum:infty}
	Each function $J^i(\bx) = O(\exp\{\|\bx\|^{\alpha}\})$ as $\|\bx\|\to\infty$, where $\alpha<2$.
\end{assumption}

Given $\bmu^i(t)$, let  player $i$ sample the random vector $\bxi^i(t)$ according to the multidimensional normal distribution $\EuScript N(\bmu^i(t)=(\mu^{i,1}(t),\ldots,\mu^{i,d}(t))^{\top},\sigma_t)$ with the following density function:
\begin{align}\label{eq:density}
	p^i&(\bx^i;\bmu^i(t),\sigma_{t})= \frac{1}{(\sqrt{2\pi}\sigma_{t})^{d}}\exp\left\{-\sum_{k=1}^{d}\frac{(x^{i,k}-\mu^{i,k}(t))^2}{2\sigma^2_{t}}\right\}.
\end{align}
According to the algorithm's setting, the cost value at the query point $\bxi(t)=(\bxi^1(t),\ldots,\bxi^N(t))\in \R^{Nd}$,  denoted by $J^i(t): =J^i(\bxi(t))$, is revealed to each player $i$. In the one-point setting, player $i$ then estimates her local gradient $\frac{\partial J^i}{\partial \bmu^i}$ evaluated at the point of the joint state $\bmu(t)=(\bmu^1(t),\ldots,\bmu^N(t))$ as follows:
\begin{align}\label{eq:est_Gd}
	\mb^i_1(t) = { J^i(t)}\frac{{\bxi^i(t)} -\bmu^i(t)}{\sigma^2_t}.
\end{align}

In the two point setting at each iteration $t$, each player $i$ makes two queries: a query corresponding to playing randomly chosen  $\bxi^i(t)$, $i\in[N]$; and another query of the cost function at $\bmu^i(t)$, $i\in[N]$. Hence, there is an extra piece of information available to each player, namely the cost function value at the mean (state) vector $\bmu(t)$: ${J^i}^0(t): = J^i(\bmu(t)) $. Then each player uses the following estimation of the local gradient $\frac{\partial J^i}{\partial \bmu^i}$ at the point  $\bmu(t)$:
	\begin{align}\label{eq:est_Gd2}
		\mb^i_2(t) = (J^i(t) - {J^i}^0(t))\frac{{\bxi^i(t)} -\bmu^i(t)}{\sigma^2_t}.
	\end{align}

\begin{rem}
Observe that the both estimations $\mb^i_j(t)$, $j=1,2$, above can be performed on the feasible set $\Ab$. One can set $\mb^i_1(t) = J^i(\Proj_{\Ab}\bxi(t))\frac{{\bxi^i(t)} -\bmu^i(t)}{\sigma^2_t}$ $\mb^i_2(t) = (J^i(\Proj_{\Ab}\bxi(t)) - J^i(\bmu(t))\frac{{\bxi^i(t)} -\bmu^i(t)}{\sigma^2_t}$, using, thus, cost values at feasible actions, namely at the points $\Proj_{\Ab}\bxi(t)$ and $\bmu(t)\in\Ab$. However, to guarantee convergence of the algorithm to the Nash equilibrium, an adjust of the updates in~\eqref{eq:alg} is required. An extra parameter needs to be introduced to project the $\bmu(t)$'s on a shrunk set and this parameter has to be balanced with both the step size $\gamma_t$ and the variance $\sigma_t$, see  \cite{arxivTatKam_Feb2022} for details of this analysis in the case of monotone games, and \cite{flaxman2005online} for similar consideration in zeroth-order online optimization. 
Lastly, note that without feasible queries, the approach can be thought of as offline learning rather than online approach. 
\end{rem}


\subsubsection{Properties of the  gradient estimators}
We provide insight into the procedure defined by Equation \eqref{eq:alg}  by deriving an analogy to a  stochastic gradient algorithm. Denote 
\begin{align}\label{eq:densityfull}
p( \bx; \bmu, \sigma)=\prod_{i=1}^Np^i(x^{i,1},\ldots,x^{i,d};\bmu^i,\sigma)
\end{align}
as the  density function of the joint distribution of players' query points $\bxi$, given some state $\bmu =(\bmu^1,\ldots,\bmu^N)$. For any $\sigma > 0$ and $i\in[N]$ define $ \tilde{J}_i : \R^{Nd} \rightarrow \R$ as
\begin{align}
\label{eq:mixedJ}
\tilde{J}_i &(\bmu, \sigma)= \int_{\mathbb R^{Nd}}J^i(\bx)p( \bx; \bmu, \sigma)d\bx.
\end{align}
Thus, $\tilde{J}_i$, $i\in[N]$, is the $i$th player's cost function in the mixed strategies, where the strategies are sampled from the Gaussian distribution with the density function in~\eqref{eq:densityfull}.
For $i\in[N]$ define $\tilde{\Mb}^{i,\sigma} (\cdot)=(\tilde M^{i,1,\sigma}(\cdot), \ldots, \tilde M^{i,d,\sigma}(\cdot))^{\top}$
as the $d$-dimensional mapping with the following elements:
\begin{align}\label{eq:mapp2}
\tilde M^{i,k,\sigma} (\bmu)=\frac{\partial {\tilde J^i(\bmu, \sigma)}}{\partial \mu^{i,k}}, \mbox{ for $k\in[d]$}.
\end{align}
Furthermore, let   $\Rb_{j}^{i}$, $j=1,2$, denote:
\begin{align}
&\Rb_{j}^{i}(\bxi(t),\bmu(t),\sigma_t) = {\mb}_j^i(t) - \tilde{\Mb}^{i,(t)} (\bmu(t),\sigma_t), \, \mbox{where}\label{eq:Rterm} \cr
&{\mb}_j^i(t) =\begin{cases}
	J^i(\bxi(t) )\frac{\bxi^i(t) -\bmu^i(t)}{\sigma^2_t}, \, \mbox{if $j=1$},\\
	(J^i(\bxi(t)) - J^i(\bmu(t)) )\frac{\bxi^i(t) -\bmu^i(t)}{\sigma^2_t}, \mbox{ if $j=2$},
	\end{cases}
\end{align}
where, to simplify notations, we used $\tilde{\Mb}^{(t)} = \tilde{\Mb}^{\sigma_t}$.
With the above definitions, the update rule \eqref{eq:alg}  is equivalent to:
\begin{align}
\label{eq:pbavmu}
&\bmu^i(t+1) =\Proj_{A^i}[\bmu^i(t) -\gamma_t\big(\tilde\Mb^{i,(t)}(\bmu(t))\cr &\qquad+\Rb_{j}^{i}(\bxi(t),\bmu(t),\sigma_t)\big)].
\end{align}
Recall that the cases $j=1$ and $j=2$ above correspond to the one-point and two-point gradient estimators, respectively.

We now show that $\tilde\Mb^{i,(t)}$ is equal to $\Mb^i$ in expectation and the term $\Rb_{j}^{i}$ has a zero-mean. Thus,  we can interpret   \eqref{eq:alg} as a stochastic gradient descent procedure.

Let $\EuScript F_{t}$ be the $\sigma$-algebra generated by the random variables $\{\bmu(k),\bxi(k)\}_{k\le t}$. First, we demonstrate in the next lemma that the mapping $\tilde{\Mb}^{(t)} = (\tilde{\Mb}^{1,(t)},\ldots,\tilde{\Mb}^{N,(t)})$ evaluated at $\bmu(t)$ is equivalent to the  pseudo-gradient in mixed strategies, that is,
\begin{align}\label{eq:gradmix}
\tilde{\Mb}^{i,(t)} (&\bmu(t))=\int_{\mathbb R^{Nd}}{\Mb^i} (\bx)p(\bx;\bmu(t),\sigma_t)d\bx.
\end{align}
Moreover, this lemma proves that  the conditional expectation of the terms defined in \eqref{eq:est_Gd} and \eqref{eq:est_Gd2}, namely,  \[{\mb}^i(t) = ({m}_j^{i,1}(t), \ldots, {m}_j^{i,d}(t))\in\R^d, \quad j = 1, 2,\] is equal to $\tilde{\Mb}_i$. 
\begin{lem}\label{lem:sample_grad}
Given Assumptions\r\ref{assum:convex} and~\ref{assum:infty}, for $j=1,2$,
\begin{align}\label{eq:deriv}
	&\tilde M^{i,k,(t)} (\bmu(t))\cr&=\E\{M^{i,k}(\bxi^1,\ldots,\bxi^N)|\xi^{i,k}\sim \EuScript N(\mu^{i,k}(t),\sigma_t), i\in[N], k\in[d]\} \cr
	&=\E\{{m}_j^{i,k}(t)|\EuScript F_t\}.
\end{align}
\end{lem}
The proof of this result is very similar to that of Lemma 1 in \cite{tatarenko2019learning} for the one-point feedback, and its extension to \cite{arxivTatKam_Feb2022} for two-point feasible feedback. 
\begin{align}
\label{eq:mathexp2}
\Rb_{j}^{i}(\bxi(t),\bmu(t),\sigma_t) = {\mb}_j^i(t) - \E\{{\mb}_j^i(t)|\EuScript F_t\},\; \;i\in[N].
\end{align}
For the sake of notation simplicity, let us use $\Rb(t) = \Rb(\bxi(t), \bmu(t),\sigma_t)$.
Our second lemma below characterizes the variance of the term  $\Rb(t)$.
\begin{lem}\label{lem:Rsq}
Under Assumptions~\ref{assum:convex} and~\ref{assum:infty}, as $\s_t\to 0$, for $j=1,2$
\[\E\{\|\Rb_{j}^{i}(t)\|^2 | \EuScript F_t\} = \begin{cases}
	O\left(\frac{d}{\sigma_t^2}\right), \, &\mbox{ if $j=1$},\\
	O(Nd^2). \, &\mbox{ if $j=2$}.
\end{cases}\]
\end{lem}
The proof of this result is similar to one of Lemma 1 in \cite{tatarenko2019learning}.


\section{Convergence Rate of the Algorithm}\label{sec:rate}
We will provide the analysis of Algorithm~\ref{alg:algorithm1} in the cases $j=1,2$ (one-point and two-point gradient estimations) under the following smoothness assumption. 
\begin{assumption}	\label{assum:Lipschitz}
The pseudo-gradient $\Mb:\R^{Nd}\to\R^{Nd}$, (see Definition~\ref{def:pg}) is twice  differentiable over $\R^{Nd}$.
\end{assumption}
Note that twice  differentiability  implies that the Jacobian of the pseudo-gradient is Lipschitz continuous. This latter condition was employed  in \cite{drusvyatskiy2022improved} for deriving the convergence rate under the strong montonicity assumption. As we do not assume  strong monotonicity, we need the slightly stronger assumption of twice differentiability (see Assumption~\ref{assum:Lipschitz} above).

\begin{theorem}\label{th:main}
Let the states $\bmu^i(t)$, $i\in[N]$, evolve according to Algorithm~\ref{alg:algorithm1} with the gradient estimators $\mb^i_j(t)$, $j=1,2$. 
Let Assumptions\r\ref{assum:convex}--\ref{assum:Lipschitz} hold.
Moreover, let the step size parameter in the procedure be chosen as follows: $\gamma_t = \frac{c}{t}$ with $c\ge \frac{1}{\nu}$, where $\nu$ is the  constant from Assumption~\ref{assum:CG_grad}. Moreover, let
\begin{align*}
	\sigma_t =\begin{cases}
	 &\frac{a}{t^{\frac{1}{4}}}, \mbox{ if $j=1$,}\\
	 &\frac{a}{t^s},  \mbox{ if $j=2$,}
	 \end{cases} 
 \end{align*}
where $a>0$ and $s\ge 1$. 

Then the joint state $\bmu(t)$ converges almost surely to the unique Nash equilibrium $\bmu^*=\ba^*$ of the game $\Gamma$, whereas the joint query point $\bxi(t)$ converges in probability to $\ba^*$. Moreover,
\begin{align*}
\E \|\bmu(t) - \ba^*\|^2=\begin{cases}
	&O\left(\frac{Nd}{t^{1/2}}\right), \mbox{ if $j=1$,}\\
	&O\left(\frac{N^2d^2}{t}\right),  \mbox{ if $j=2$.}
\end{cases}
\end{align*}
\end{theorem}

\subsection{Discussion on the established rates}
\subsubsection{Algorithm parameters} The optimal step size $\gamma_t$, in both  one- and two-point estimation procedures, is dependent on the monotonicity constant $\nu$. For example, for $j=1$, $\gamma_t = \frac{c}{t}$ with $c\ge \frac{1}{\nu}$. This choice is motivated by the idea of applying the Chung's lemma for deriving the rates (see the past paragraph in the proof of the main result). Any other setting of the following form $\gamma_t = \frac{c}{t}$, $c<\frac{1}{\nu}$, will imply applicability of the Chung's lemma as well. However, in this case, according to the Chung's lemma, we get a less tight convergence rate. %

\subsubsection{Comparison with optimal rates for strongly monotone games}
In the one-point setting, the best rate for learning Nash euqilibria in strongly monotone games \cite{drusvyatskiy2022improved} was derived as $O\left(\frac{N^2d^2}{t^{1/2}}\right)$ (see also \cite{arxivTatKam_Feb2022} on the rate estimation without explicit dependence on the dimension). Here, we achieve the same rate without the monotonicity assumption and with a better dependence on the game's dimension. In the two-point setting, the best rate was established as $O\left(\frac{1}{t}\right)$ in \cite{arxivTatKam_Feb2022} under strong monotonicity of the game. Our current result matches this as well. 

\subsubsection{Comparison with optimal rates for strongly convex optimization} Our rates derived here match (as functions of $t$) those of zeroth-order optimization for strongly convex smooth problems (see \cite{Duchi2015, shamir2013complexity}) under one-point and two-point feedback. As this class of optimization problems is a strict subclass of game-theoretic problems considered here, our results are notable. In particular, we can achieve these rates in the optimization setting without  assuming convexity of the problem. For concreteness, observe that the game provided in Example \ref{example:3player} is a potential game since its pseudo-gradient is symmetric. This game has a unique Nash equilibrium and thus, finding the equilibrium is equivalent to finding the optimizer of the non-convex potential function, with the special property of strong variational stability of the optimizer. We note that in the case of non-convex optimization, several conditions have been derived, such as gradient dominance \cite{karimi2016linear}, that extend the applicability of optimal rates of gradient descent. It will be highly relevant for the future research to explore the connection of the above optimization conditions with the VS equilibrium definition.

\subsection{Proof of main result}

We will base our analysis on the algorithm's representation in~\eqref{eq:pbavmu}. Thus, in this subsection we exploit the properties of the term $\tilde{\Mb}^{i,(t)} (\bmu(t))$ therein.

Let us now focus on the mapping $\tilde{\Mb}^{(t)}(\cdot) =(\tilde{\Mb}^{1,(t)} (\cdot), \ldots, \tilde{\Mb}^{N,(t)}(\cdot))$, where, as before, for any $\bmu\in\R^{Nd}$
\begin{align}
\label{eq:smooth_t_game}
&\tilde{\Mb}^{i,(\sigma_t)} (\bmu) =\tilde{\Mb}^{i,(t)} (\bmu)\cr
&= \nabla_{\bmu^i}J^i(\bmu,\sigma_t)= \int_{\mathbb R^{Nd}}{\Mb^i} (\bx)p(\bx;\bmu,\sigma_t)d\bx
\end{align}
for some given $\sigma_t$ (see definition in \eqref{eq:mapp2} and also the property \eqref{eq:gradmix}).  We emphasize that such mapping is the pseudo-gradient in the mixed strategies, given that the joint action is generated by the normal distribution with the density~\eqref{eq:densityfull}.

A technical novelty in deriving the convergence in the absence of strong monotonicity is Proposition \ref{prop:strMon} below. It states that under the made assumption the Nash equilibrium $\ab^*$ stays ``almost" strongly variationally stable with respect to the pseudo-gradient in the mixed strategies. 
\begin{prop}\label{prop:strMon}
Let Assumptions~\ref{assum:compact},~\ref{assum:CG_grad},~\ref{assum:infty}, and \ref{assum:Lipschitz} hold.  Then $\langle\tilde{\Mb}^{(t)}(\bmu), \bmu-\ab^*\rangle \ge -O(Nd\sigma_t^2)  + \nu\|\bmu - \ab^*\|^2$ for any $\bmu\in\Ab$.
\end{prop}
\begin{proof}
We focus on each term in the following sum representation of the dot-product $\langle\tilde{\Mb}^{(t)}(\bmu), \bmu-\ab^*\rangle$:
\begin{align}\label{eq:sum}
\langle\tilde{\Mb}^{(t)}(\bmu), \bmu-\ab^*\rangle = \sum_{i=1}^{N}\sum_{k=1}^{d} \tilde{M}^{i,k,(t)}(\bmu)(\mu^{i,k}-a^{*,i,k}).
\end{align}
Due to Assumption~\ref{assum:Lipschitz}, we can use the following Taylor approximation for the elements $M^{i,k}$, $i\in[N]$, $k\in[d]$, of the mapping $\Mb$ around some point $\bmu\in\Ab$ (see Definition~\ref{def:pg}): 
\begin{align}\label{eq:Ta}
M^{i,k}(\bx) = &M^{i,k}(\bmu) + \langle\nabla M^{i,k}(\bmu), \bx - \bmu\rangle \cr
&+ \langle\nabla^2 M^{i,k}(\tilde \bx) (\bx - \bmu),\bx - \bmu\rangle,
\end{align}
where $\tilde \bx = \bmu + \theta(\bx - \bmu)$ for some $\theta\in[0,1]$. 
Taking into account the fact that 
\[\tilde{M}^{i,k,(t)}(\bmu) = \int_{\mathbb R^{Nd}}{M^{i,k}} (\bx)p(\bx;\bmu,\sigma_t)d\bx,\]
we obtain 
\begin{align}
&\tilde{M}^{i,k,(t)}(\bmu)(\mu^{i,k}-a^{*,i,k})\cr
& = \left[\int_{\mathbb R^{Nd}}({M^{i,k}} (\bx) - M^{i,k} (\bmu))p(\bx;\bmu,\sigma_t)d\bx\right](\mu^{i,k}-a^{*,i,k})\cr
&\qquad +  M^{i,k} (\bmu)(\mu^{i,k}-a^{*,i,k})\cr
& = \left[\int_{\mathbb R^{Nd}}\langle\nabla^2 M^{i,k}(\tilde \bx) (\bx - \bmu),\bx - \bmu\rangle p(\bx;\bmu,\sigma_t)d\bx\right]\cr
&\qquad\qquad\quad\times(\mu^{i,k}-a^{*,i,k}) +  M^{i,k} (\bmu)(\mu^{i,k}-a^{*,i,k}),
\end{align}
where in the last equality we used~\eqref{eq:Ta} and the fact that
\[\int_{\mathbb R^{Nd}}\langle\nabla M^{i,k}(\bmu), \bx - \bmu\rangle p(\bx;\bmu,\sigma_t)d\bx=0.\]
We note that $\int_{\mathbb R^{Nd}}\langle\nabla^2 M^{i,k}(\tilde \bx) (\bx - \bmu),\bx - \bmu\rangle p(\bx;\bmu,\sigma_t)d\bx = \E\{\langle \nabla^2 M^{i,k}(\tilde \bxi) (\bxi - \bmu),\bxi - \bmu\rangle\}$, given that $\bxi$ has the Gaussian distribution with the density function $p(\bx;\bmu,\sigma_t)$ and $\tilde \bxi = \bmu + \theta(\bxi - \bmu)$. Next, using the  Cauchy-Schwarz and H\"older's inequalities as well as Lemma~\ref{lem:aux} (see Appendix~\ref{app:Lemma}), we obtain
\begin{align}\label{eq:Ineq}
\E&\{\langle\nabla^2 M^{i,k}(\tilde \bxi) (\bxi - \bmu),\bxi - \bmu\rangle\}\cr
&\ge - \E\{\|\nabla^2 M^{i,k}(\tilde \bxi) \|\|\bxi - \bmu\|^2\} = -O(Nd\sigma_t^2). 
\end{align}
Combining~\eqref{eq:sum}-\eqref{eq:Ineq}, we conclude that 
\begin{align*}
\langle\tilde{\Mb}^{(t)}(\bmu), \bmu-\ab^*\rangle \ge &-O(Nd\sigma_t^2)\sum_{i=1}^{N}\sum_{k=1}^{d}\|\mu^{i,k}-a^{*,i,k}\|  \cr
&+ \langle\Mb(\bmu), \bmu - \ab^*\rangle\cr
& \ge - O(Nd\sigma_t^2)  + \nu\|\bmu - \ab^*\|^2, 
\end{align*}
where in the last inequality we used $\bmu,\ab^*\in\Ab$ and compactness of $\Ab$ (Assumption~\ref{assum:compact}) and $\langle\Mb(\bmu), \bmu - \ab^*\rangle\ge\nu\|\bmu - \ab^*\|^2 $ (Assumption~\ref{assum:CG_grad}). 
\end{proof}

We are now equipped to provide the proof of Theorem \ref{th:main}.

\begin{proof}{(of Theorem \ref{th:main})}
Let us notice that due to the theorem's conditions and the particular choice $\sigma_t\to 0$, as $t\to\infty$, Proposition\ref{prop:strMon} hold.

We consider $\|\bmu(t+1)-\ab^*\|^2$.
We aim to bound the growth of $\|\bmu(t+1)-\ab^*\|^2$ in terms of $\|\bmu(t)-\ab^*\|^2$ and, thus, to obtain the convergence rate of the sequence  $\|\bmu(t+1)-\ab^*\|^2$. 

We analyze each term in the following sum $\|\bmu(t+1)-\ab^*\|^2 = \sum_{i=1}^{N} \|\bmu^i(t+1)-\ab^{i*}\|^2$.
From the procedure for the update of $\bmu(t)$ in ~\eqref{eq:pbavmu}, the fact that $\ab^{i*}\in A^i$ and the non-expansion property of the projection operator, we obtain that for any $i\in[N]$
\begin{align}\label{eq:nonexp}
\|&\bmu^i(t+1)-\ab^{i*}\|^2 \le  \|\bmu^i(t)-\ab^{i*}\|^2 \cr
&\qquad- 2\gamma_t\langle\tilde{\Mb}^{i,(t)}(\bmu(t)), \bmu^i(t)-\ab^{i*}\rangle \cr
&\qquad-2\gamma_t\langle\Rb_{j}^{i}(t), \bmu^i(t)-\ab^{i*}\rangle + \gamma^2_t\|\Gb^i_j(t)\|^2,
\end{align}
where, for ease of notation, we have defined $\Rb_{j}^{i}(t) = \Rb_{j}^{i}(\bxi(t), \bmu(t),\sigma_t)$ and
\begin{align}\label{eq:G_1}
\Gb_{i}^{j}(t) = &\tilde{\Mb}^{i,(t)}(\bmu(t))  +\Rb_{j}^{i}(t).
\end{align}
We expand $\Gb_{i,j}$ as below and bound the terms in the expansion.
\begin{align}\label{eq:G_1_norm}
&\|\Gb_{i}^{j}(t)\|^2 = \|\tilde{\Mb}^{i,(t)}(\bmu(t)) \|^2 \cr
&+ \|\Rb_{j}^{i}(t)\|^2 +2\langle\tilde{\Mb}^{i,(t)}(\bmu(t)) ,\Rb_{j}^{i}(t)\rangle.
\end{align}
Thus, by taking into account~\eqref{eq:mathexp2}, which implies $\E\{\Rb_{j}^{i}(t)|\EuScript F_t\} =\boldsymbol 0$ for any $\bmu$, and the Cauchy-Schwarz inequality, we get from \eqref{eq:nonexp} 
\begin{align}\label{eq:firstEst}
&\E\{\|\bmu^i(t+1)-\ab^{i*}\|^2|\EuScript F_t\} \le \|\bmu^i(t)-\ab^{i*}\|^2 \cr
&- 2\gamma_t\langle\tilde{\Mb}^{i,(t)}(\bmu(t)), \bmu^i(t)-\ab^{i*}\rangle \cr
& +\gamma^2_t\E\{\|\Gb^{i}_j(t)\|^2|\EuScript F_t\}\cr
& \le\|\bmu^i(t)-\ab^{i*}\|^2 \cr
&- 2\gamma_t\langle\tilde{\Mb}^{i,(t)}(\bmu(t)), \bmu^i(t)-\ab^{i*}\rangle \cr
& + \gamma^2_t[\|\tilde{\Mb}^{i,(t)}(\bmu(t)) \|^2+\E\{\|\Rb_{j}^{i}(t)\|^2|\EuScript F_t\}].
\end{align}
Next, taking into account Lemmas~\ref{lem:Rsq} and using continuity of $\tilde{\Mb}^{(t)}$ (for the proof of continuity see \cite{arxivTatKam_Feb2022}, Proposition 2 therein) as well as compactness of $\Ab$, and, thus, boundedness of $\|\tilde{\Mb}^{i,(t)}(\bmu(t))\|$, we conclude that 
\begin{align}\label{eq:secondEst}
&\E\{\|\bmu^i(t+1)-\ab^{i*}\|^2|\EuScript F_t\} \le\|\bmu^i(t)-\ab^{i*}\|^2 \cr
&\,- 2\gamma_t\langle\tilde{\Mb}^{i,(t)}(\bmu(t)), \bmu^i(t)-\ab^{i*}\rangle+ h_0(t),
\end{align}
where
\begin{align}\label{eq:ht0}
h_0(t) = \begin{cases}
	&O\left(\frac{d\gamma^2_t}{\sigma^2_t}\right), \mbox{ if $j=1$},\\
	&O({Nd^2\gamma^2_t}), \mbox{ if $j=2$}.
\end{cases}	
\end{align}
Thus, due to Proposition~\ref{prop:strMon} we conclude from~\eqref{eq:secondEst} by summing up the inequalities over $i=1,\ldots,N$,
\[\E\{\|\bmu(t+1)-\ab^{*}\|^2|\EuScript F_t\}\le(1-\nu\gamma_t)\|\bmu(t)-\ab^*\|^2+h_1(t),\]
where 
\begin{align}\label{eq:ht}
h_1(t) = \begin{cases}
	&O\left(\frac{Nd\gamma^2_t}{\sigma^2_t}+ Nd\gamma_t\sigma_t^2\right), \mbox{ if $j=1$},\\
	&O({N^2d^2\gamma^2_t} + Nd\gamma_t\sigma_t^2), \mbox{ if $j=2$}.
\end{cases}	
\end{align}
Thus, given the settings for the parameters $\gamma_t = \frac{c}{t}$ with $c\ge \frac{1}{\nu}$ and $\sigma_t$, the definition of $h_1(t)$ (see~\eqref{eq:ht}),  we conclude that
\begin{align}\label{eq:LV2}
&\E\{\|\bmu(t+1)-\ab^*\|^2|\EuScript F_t\}\cr&\le\left(1-\frac{1}{t}\right)\|\bmu(t)-\ab^*\|^2
+ H(t),
\end{align}
where
\begin{align}\label{eq:Ht}
	H(t) = \begin{cases}
		&O\left(\frac{Nd}{t^{3/2}}\right), \mbox{ if $j=1$},\\
		&O(\frac{N^2d^2}{t^2}), \mbox{ if $j=2$}.
	\end{cases}	
\end{align}
The inequality \eqref{eq:LV2} implies that $\bmu(t)$ converges to $\ab^*$ almost surely (see Lemma 10 in Chapter 2.2~\cite{Polyak}). Taking into account that $\bxi(t)\sim\EuScript N(\bmu(t),\sigma_t)$ and $\sigma_t\to 0$ as $t\to\infty$, we conclude that $\bxi(t)$ converges weakly to a Nash equilibrium $\ab^*$. Moreover, according to the Portmanteau Lemma \cite{portlem}, this convergence is also in probability.
Next, by taking the full expectation of the both sides in~\eqref{eq:LV2}, we obtain
\begin{align*}
\E[\|\bmu(t+1)-\ab^*\|^2]\le&\left(1-\frac{1}{t}\right)\E\|\bmu(t) - \ab^*\|^2\cr&+ H(t).
\end{align*}
By applying the Chung's lemma (see Lemma 4 in Chapter 2.2~\cite{Polyak}, for the reader's convenience we provide this lemma in Appendix~\ref{app:Ch})) to the inequality above, we conclude that
\begin{align*}
\E \|\bmu(t) - \ba^*\|^2=\begin{cases}
	&O\left(\frac{Nd}{t^{1/2}}\right), \mbox{ if $j=1$,}\\
	&O\left(\frac{N^2d^2}{t}\right),  \mbox{ if $j=2$.}
\end{cases}
\end{align*}
\end{proof}
\begin{figure}[!t]
\centering
\includegraphics[width=\linewidth]{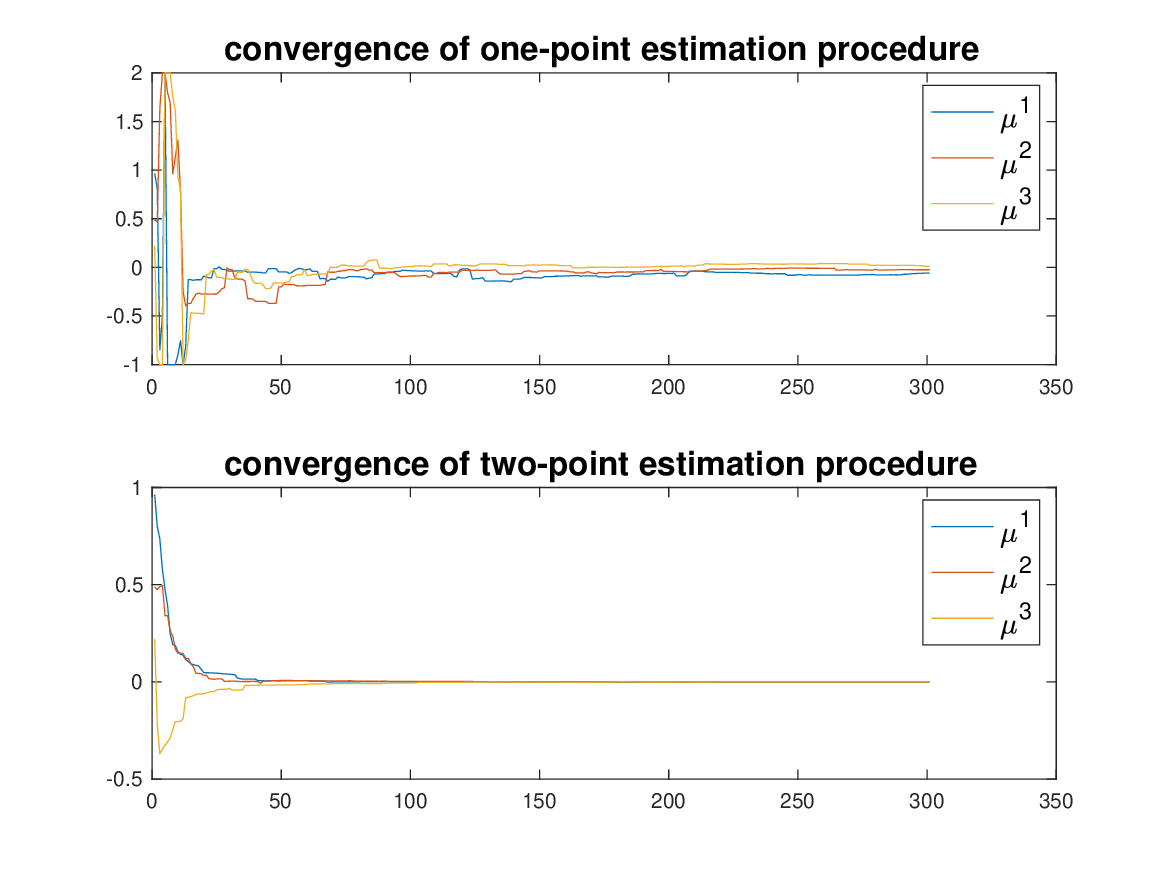}
\caption{Convergence of the algorithm for the game of Example \ref{example:3player} with strategy sets as $[-1,2]$. }
\label{fig:cube}
\end{figure}

\section{Numerical example}\label{sec:simulation}
We consider Example \ref{example:3player} in Section \ref{sec:problem}. Notice that the game satisfies Assumptions\r\ref{assum:convex}--\ref{assum:Lipschitz}. In the plots below, we show the convergence of the algorithm using the one-point and two-point feedback. The first plot considers the strategy set $[-1,2]$, whereas the second plot considers the strategy set $[0,1]$. The parameters were set to $c = 1$, $a = 1$ in both one-point and two-point setting, and $s=1$ for the case of the two-point setting.  The initial state $\bmu(0)$ was chosen from standard normal distribution. As predicted by the theory, in both one-point and two-point settings, the proposed algorithms converge to the unique SVS equilibrium of the game, despite lack of monotonicity. The two-point estimation results in much faster convergence rate, as also predicted.

\section{Conclusion}
\label{sec:conclusion}
We established  the convergence rate of the zeroth-order gradient play to the strongly VS Nash equilibrium of a convex game. In both the one-point and two-point setting, our rates of  $O\left(\frac{Nd}{t^{1/2}}\right)$ and $O\left(\frac{N^2d^2}{t}\right)$ appear to be optimal (as functions on $t$) as they match the best rates established for the subclass of strongly monotone games. An open question is the lower bound for the convergence rate of zeroth-order learning in convex games with respect to the problem dimension. Moreover, it will be interesting to further relax the assumptions so as to establish the convergence and derive the convergence rate of a zeroth-order gradient play to a merely VS equilibrium or to equilibria in non-convex games. 

\begin{figure}[!t]
\centering
\includegraphics[width=\linewidth]{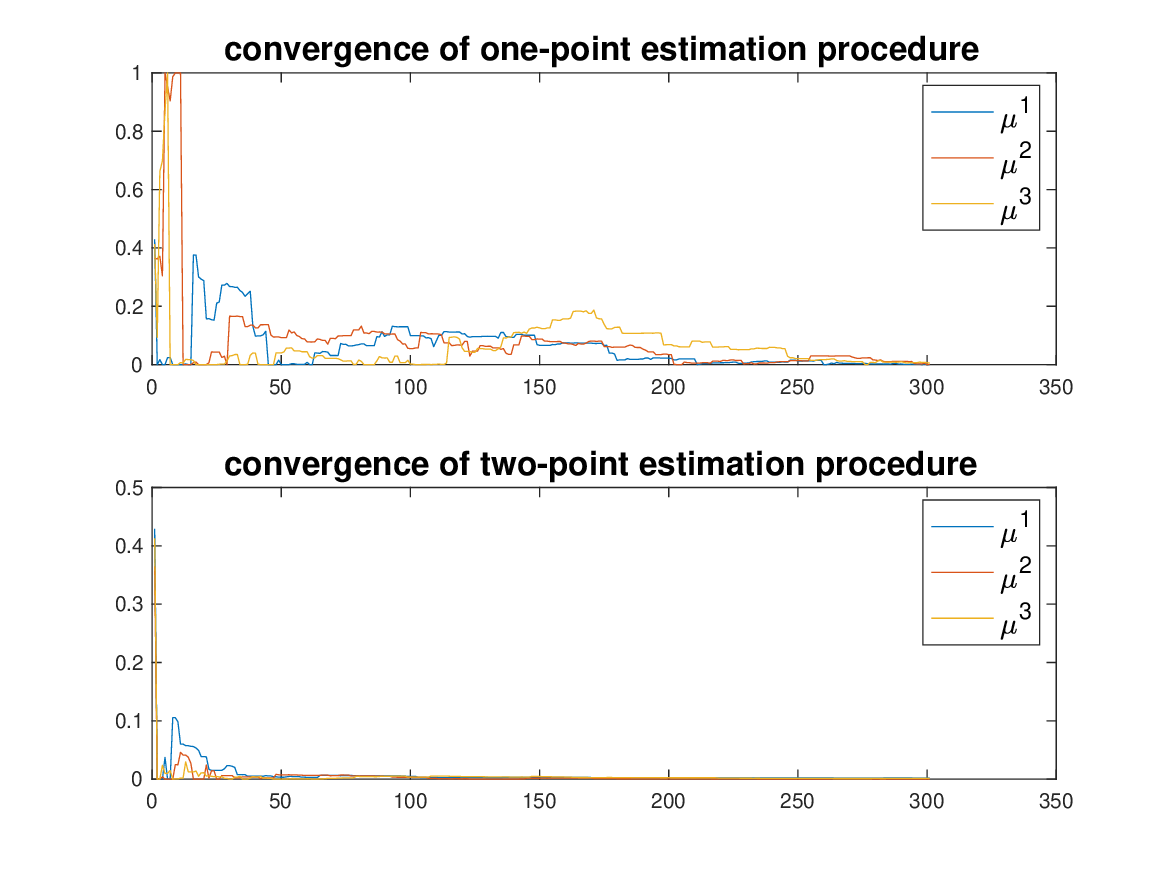}
\caption{Convergence of the algorithm for the game of Example \ref{example:3player} with strategy sets as $[0,1]$. }
\label{fig:cube_boundary}
\end{figure}

\bibliographystyle{plain}
\bibliography{srtrMonGames_ref}

%
%


\appendix
\subsection{Auxiliary Result}\label{app:Lemma}
The following auxiliary lemma will be used in the proofs of some propositions below.
\begin{lem}\label{lem:aux}
Let some continuous function $f(\bx): \R^{Nd}\to\R$ be  such that $f(\bx)\ge 0$ for any $\bx\in\R^{Nd}$ and  $f(\bx) = O(\exp\{\|\bx\|^{\alpha}\})$ as $\|\bx\|\to\infty$, where $\alpha<2$. Let $\Ab$ be some compact subset of $\R^{Nd}$. Finally, let $p(\bx; \bmu, \sigma_t)$, $\bmu\in\Ab$, be the density function of the Gaussian vector $\bxi$ as defined in~\eqref{eq:densityfull} and $\tilde \bxi = \bmu+\theta(\bxi-\bmu)$ for some $\theta\in[0,1]$.
Then there exists a constant $C>0$  such that $\E\{f(\tilde \bxi)\}\le 2^{Nd}C$.
Moreover, if $\sigma_t\to 0$ as $t\to\infty$, then there exists $c>0$, which is independent on $d$ and $N$, such that $\E\{f(\tilde \bxi)\}\le c$ for all sufficiently large $t$.
\end{lem}
{Proof can be found in \cite{arxivTatKam_Feb2022} (see Lemma 5 therein). }


%

\subsection{The Chung's Lemma (Lemma 4 in Chapter 2.2. \cite{Polyak})}\label{app:Ch}
\begin{lem}
Let $u_k\ge 0$ and
\[u_{k+1}\le\left(1-\frac{c}{k}\right)u_k + \frac{d}{k^{1+p}}, \quad d,p,c>0.\]
Then
\begin{align*}
	u_k\le d(c-p)^{-1}k^{-p} + o(k^{-p}), \quad &\mbox{if $c>p$},\cr
	u_k = O(k^{-c}\ln k), \quad &\mbox{if $c=p$}, \cr
	u_k = O(k^{-c}), \quad &\mbox{if $c<p$}.
\end{align*}
\end{lem}

%

\end{document}